\pdfoutput=1
\documentclass{amsart}
\usepackage[style=alphabetic,maxalphanames=4]{biblatex}
\usepackage[paperheight=279mm,paperwidth=18cm,textheight=26cm,textwidth=14cm,includehead]{geometry}
\usepackage[utf8]{inputenc}
\usepackage{amsfonts,amssymb,amsmath,amsthm}
\usepackage{mathtools}
\usepackage{etoolbox}
\usepackage{mathabx,accents,mathdots}
\usepackage{esint}
\usepackage{xcolor} 

\usepackage[unicode]{hyperref}
\hypersetup{
bookmarks=true,
colorlinks=true,
citecolor=[rgb]{0,0,0.5},
linkcolor=[rgb]{0,0,0.5},
urlcolor=[rgb]{0,0,0.75},
pdfpagemode=UseNone,
pdfstartview=FitH,
pdfdisplaydoctitle=true,
pdflang=en-US
}

\def\PZdefchar#1{
  \expandafter\def\csname frak#1\endcsname{\mathfrak{#1}}
  \expandafter\def\csname bf#1\endcsname{\mathbf{#1}}
  \expandafter\def\csname scr#1\endcsname{\mathcal{#1}}
  \expandafter\def\csname cal#1\endcsname{\mathcal{#1}}}
\def\PZdefloop#1{\ifx#1\PZdefloop\else\PZdefchar#1\expandafter\PZdefloop\fi}
\PZdefloop abcdefghijklmnopqrstuvwxyzABCDEFGHIJKLMNOPQRSTUVWXYZ\PZdefloop

\newcommand{\R}{\mathbb{R}}
\newcommand{\N}{\mathbb{N}}

\newcommand{\dif}{\mathrm{d}}

\newcommand{\one}{\mathbf{1}}
\newcommand{\dist}{\operatorname{dist}}

\numberwithin{equation}{section}
\newtheorem{theorem}{Theorem}
\numberwithin{theorem}{section}

\newtheorem{lemma}[theorem]{Lemma}
\newtheorem{corollary}[theorem]{Corollary}
\theoremstyle{definition}
\newtheorem{definition}[theorem]{Definition}

\theoremstyle{remark}
\newtheorem{remark}[theorem]{Remark}

\DeclarePairedDelimiter\abs{\lvert}{\rvert}

\DeclarePairedDelimiter\norm{\lVert}{\rVert}

\providecommand\given{}
\newcommand\SetSymbol[1][]{%
\nonscript\:#1\vert
\allowbreak
\nonscript\:
\mathopen{}}
\DeclarePairedDelimiterX\Set[1]\{\}{\renewcommand\given{\SetSymbol[\delimsize]}#1}
\DeclarePairedDelimiterXPP\EE[1]{\E}{\lparen}{\rparen}{}{\renewcommand\given{\SetSymbol[\delimsize]}#1} 
\DeclarePairedDelimiterX\inn[2]{\langle}{\rangle}{#1,#2}

\makeatletter
\newcommand\@avsum[2]{%
  {\sbox0{$\m@th#1\sum$}%
   \vphantom{\usebox0}%
   \ooalign{%
     \hidewidth
     \smash{\vrule height\dimexpr\ht0+1pt\relax depth\dimexpr\dp0+1pt\relax}%
     \hidewidth\cr
     $\m@th#1\sum$\cr
   }%
  }%
}
\newcommand{\avsum}{\mathop{\mathpalette\@avsum\relax}\displaylimits}
\newcommand\@avprod[2]{%
  {\sbox0{$\m@th#1\prod$}%
   \vphantom{\usebox0}%
   \ooalign{%
     \hidewidth
     \smash{\vrule height\dimexpr\ht0+1pt\relax depth\dimexpr\dp0+1pt\relax}%
     \hidewidth\cr
     $\m@th#1\prod$\cr
   }%
  }%
}
\newcommand{\avprod}{\mathop{\mathpalette\@avprod\relax}\displaylimits}
\newcommand{\@avL}[2]{%
\ooalign{{$\m@th#1\mbox{--}$}\cr {$\m@th#1 L$}\cr}}
\newcommand{\avL}{\mathpalette\@avL\relax}
\newcommand{\@avell}[2]{%
\ooalign{{$\m@th#1\mbox{--}$}\cr {$\m@th#1 \ell$}\cr}}
\newcommand{\avell}{\mathpalette\@avell\relax}
\newcommand{\@avD}{%
  \ooalign{{$\mathrm{D}$}\cr \hidewidth\raise.2ex\hbox{$\vert$}\hidewidth\cr}}
\newcommand{\avDec}{\@avD\mathrm{ec}}
\makeatother
\newcommand{\Dec}{\mathcal{D}} 
\newcommand{\BilinDec}{\mathcal{B}} 

\DeclareMathOperator{\lin}{lin}

\DeclareMathOperator{\supp}{supp}

\newcommand{\Part}[2][]{\calP(\ifstrempty{#1}{}{#1,}#2)} 
\newcommand{\Uspace}{\calU^\circ} 

\def\dist{\operatorname{dist}}
\def\wb1{w_{B_{\delta^{-1}}}}

\numberwithin{equation}{section}

\newcommand{\lsm}[0]{\lesssim}

\newcommand{\mc}[1]{\mathcal{#1}}
\newcommand{\ov}[1]{\overline{#1}}

\newcommand{\E}[0]{\mathcal{E}}

\newcommand\numberthis{\addtocounter{equation}{1}\tag{\theequation}}

\definecolor{orange}{rgb}{1,0.5,0}

\hypersetup{
pdftitle={ℓ² decoupling for the moment curve},
pdfauthor={Guo, Li, Yung, Zorin-Kranich},
}

\begin{document}
\title[$\ell^2$ decoupling for the moment curve]{A short proof of $\ell^2$ decoupling for the moment curve}
\author[S.~Guo]{Shaoming Guo}
\address[SG]{Department of Mathematics, University of Wisconsin-Madison, Madison, WI-53706, USA}
\email{shaomingguo@math.wisc.edu}
\author[Z.~Li]{Zane Kun Li}
\address[ZKL]{Department of Mathematics, Indiana University Bloomington, Bloomington, IN-47405, USA}
\email{zkli@iu.edu}
\author[P.-L.~Yung]{Po-Lam Yung}
\address[PLY]{Department of Mathematics, The Chinese University of Hong Kong, Shatin, Hong Kong \quad \textit{and} \quad Mathematical Sciences Institute, The Australian National University, Canberra, Australia}
\email{plyung@math.cuhk.edu.hk \quad \textit{and} \quad polam.yung@anu.edu.au}
\author[P.~Zorin-Kranich]{Pavel Zorin-Kranich}
\address[PZK]{Mathematical Institute\\ University of Bonn}
\email{pzorin@uni-bonn.de}

\begin{abstract}
We give a short and elementary proof of the $\ell^{2}$ decoupling inequality for the moment curve in $\hat{\R}^k$, using a bilinear approach inspired by the nested efficient congruencing argument of Wooley \cite{MR3938716}.
\end{abstract}
\maketitle

\section{Introduction}
The sharp $\ell^{2}$ decoupling inequality for the moment curve, proved by Bourgain, Demeter, and Guth \cite{MR3548534}, implies Vinogradov's mean value theorem with the optimal exponents.
The optimal exponents in Vinogradov's mean value theorem have also been obtained by Wooley \cite{MR3938716}, using a nested efficient congruencing argument.
Efficient congruencing is a method of counting the number of solutions to Diophantine systems, and counting arguments do not usually imply decoupling inequalities.
Nevertheless, in this article, we borrow insights from \cite{MR3938716} (see also Heath-Brown \cite{arxiv:1512.03272}), to give a short proof of the $\ell^2$ decoupling inequality for the moment curve, namely Theorem~\ref{thm:main} below.

Let $k \in \N$ and $\Gamma \colon [0,1] \to \hat{\R}^k$ be the moment curve in $\hat{\R}^k$ (the Pontryagin dual of $\R^{k}$, which is itself isomorphic to $\R^{k}$), parametrized by
\[
\Gamma(\xi) := (\xi,\xi^2, \dotsc, \xi^k).
\]
For $\delta \in (0,1)$, let $\Part{\delta}$ denote the partition of the interval $[0,1]$ into dyadic intervals with length $2^{-\lceil \log_2 \delta^{-1} \rceil}$.
For a dyadic interval $J$, let $\calU_{J}$ be the parallelepiped of dimensions $\abs{J}^{1} \times \abs{J}^{2} \times \dotsm \times \abs{J}^{k}$ whose center is $\Gamma(c_J)$ and sides are parallel to $\partial^{1}\Gamma(c_{J})$, $\partial^{2}\Gamma(c_{J})$, $\dotsc$, $\partial^{k}\Gamma(c_{J})$, where $c_J$ is the center of $J$.
We write $p_k := k (k+1)$ for the critical exponent, and $\norm{\cdot}_{p} := \norm{\cdot}_{L^{p}(\R^{k})}$.

\begin{definition}
\label{def:dec-const}
For $\delta \in (0,1)$, the \emph{$\ell^{2}L^{p_{k}}$ decoupling constant} $\Dec_{k}(\delta)$ for the moment curve in $\hat{\R}^k$ is the smallest number for which the inequality
\begin{equation}
\label{eq:dec-const}
\norm[\Big]{ \sum_{J \in \Part{\delta}} f_{J} }_{p_{k}}
\leq \Dec_{k}(\delta)
\Bigl( \sum_{J \in \Part{\delta}} \norm[\big]{ f_{J} }_{p_{k}}^{2} \Bigr)^{1/2}
\end{equation}
holds for any tuple of functions $(f_{J})_{J \in \Part{\delta}}$ with $\supp \widehat{f_{J}} \subseteq \calU_{J}$ for all $J$.
\end{definition}

\begin{theorem}[{\cite{MR3548534}}]
\label{thm:main}
For every $k \in \N$ and every $\epsilon>0$, there exists a finite constant $C_{k,\epsilon}$ such that
\begin{equation}
\label{eq:main}
\Dec_{k}(\delta) \le C_{k,\epsilon} \delta^{-\epsilon}, \text{ for every } \delta\in (0, 1).
\end{equation}
\end{theorem}

Strictly speaking, Theorem~\ref{thm:main} was stated in \cite{MR3548534} in a superficially weaker form, but the proof given there also yields the result as stated in Theorem~\ref{thm:main}, see \cite{arxiv:1811.02207} or \cite[Chapter 11]{DemeterBook} for more details.
It is now well-known that Theorem~\ref{thm:main} implies the following Vinogradov's mean value estimates (see \cite[Section 4]{MR3548534} for a proof):
\begin{corollary}[{\cite{MR3548534}, \cite{MR3938716}}]
\label{cor:Vinogradov-mean}
Let $k\geq 1$ and $s \geq 1$.
Then, for every $\epsilon>0$ and every $N \geq 1$, we have
\begin{equation}
\label{eq:Vinogradov-mean}
\int_{[0,1]^{k}} \abs[\Big]{ \sum_{n=1}^{N} a_{n} e(n x_{1}+\dotsb+n^{k}x_{k}) }^{2s} \dif x_{1} \dotsc \dif x_{k}
\lesssim_{k,s,\epsilon}
N^{\epsilon} (1+N^{s-k(k+1)/2}) \bigl( \sum_{n=1}^{N} \abs{a_{n}}^{2} \bigr)^{s}.
\end{equation}
Here $e(t) := \exp(2\pi i t)$ is the unit character.
\end{corollary}

The proof of Theorem~\ref{thm:main} in \cite{MR3548534} uses a multilinear variant of the decoupling inequality, whose proof relies crucially on (multilinear) Kakeya--Brascamp--Lieb type inequalities.
On the contrary, we will use a \emph{bilinear} variant of the decoupling inequality.
In our proof, the transversality that was captured in \cite{MR3548534} by Kakeya--Brascamp--Lieb type inequalities is instead exploited via introducing certain asymmetric bilinear decoupling constants.
Such bilinear decoupling constants are carefully designed to facilitate an efficient way of induction on the dimension $k$.
In fact, an averaging argument involving Fubini's theorem allows us to apply very neatly the uncertainty principle, and gain access to lower degree decoupling.
To sum up, instead of using Kakeya--Brascamp--Lieb type estimates, we will rely only on lower degree decoupling and H\"{o}lder inequalities in the induction step.

A related bilinear argument has been developed by Wooley in the context of Vinogradov mean value estimates; see \cite{MR3938716} and references therein.
For a comparison between Wooley's efficient congruencing approach and Bourgain-Demeter-Guth's decoupling approach, we refer the reader to \cite{MR3939285}.
In the context of decoupling inequalities, the bilinear approach was previously implemented for the parabola (case $k=2$ of Theorem~\ref{thm:main}) in \cite{arxiv:1805.10551} and the cubic moment curve in \cite{arxiv:1906.07989}.
Note, however, that the decoupling theorem proved in \cite{arxiv:1906.07989} is weaker than the $k=3$ case of Theorem~\ref{thm:main}; it follows from Theorem~\ref{thm:main} by estimating the $\ell^{2}$ sum on the right-hand side of \eqref{eq:dec-const} by an $\ell^{4}$ sum times $\delta^{-1/4}$.
Moreover, the method in \cite{arxiv:1906.07989} does not seem to work for degree $k\ge 4$.
The reason is exactly the same as why the arguments in \cite{arxiv:1512.03272} and \cite{MR3479572} do not generalize to the cases $k\ge 4$, which was explained at the end of Section 3 of \cite{arxiv:1512.03272}.
In short, if one follows the approach of \cite{arxiv:1512.03272} and \cite{MR3479572} in the case $k\ge 4$, then ``singular'' solutions to the Vinogradov system will start dominating and prevent an optimal estimate on the number of solutions.

\subsection*{Notation}
For a sequence of real numbers $(A_{\theta})_{\theta\in\Theta}$, we write $\ell^{2}_{\theta\in\Theta} A_{\theta} := \bigl( \sum_{\theta\in\Theta} \abs{A_{\theta}}^{2} \bigr)^{1/2}$. For $C > 0$ and a parallelepiped $\calU$, we will denote by $C \calU$ the parallelepiped similar to $\calU$, with the same center but $C$ times the side lengths.
For a dyadic interval $I$, we let $\Part[I]{\delta}$ be the partition of $I$ into dyadic intervals with length $2^{\lceil \log_2 \delta^{-1} \rceil}$.
If $\delta \in (0,1)$, $I$ is a dyadic interval of length $\geq \delta$, and a family of functions $(f_J)$ has been chosen so that $\supp \widehat{f_J} \subseteq \calU_J$ for every $J \in \Part[I]{\delta}$, then we will write $f_I:=\sum_{J\in \Part[I]{\delta}}f_J$.

\subsection*{Acknowledgements}
SG and ZL would like to thank the Department of Mathematics at the Chinese University of Hong Kong for their kind hospitality during their visits, where part of this work was done.
SG was supported in part by the NSF grant 1800274.
ZL was supported by NSF grant DMS-1902763.
PY was partially supported by a General Research Fund CUHK14303817 from the Hong Kong Research Grants Council, and a direct grant for research from the Chinese University of Hong Kong (4053341).
PZ was partially supported by the Hausdorff Center for Mathematics (DFG EXC 2047).
The authors would also like to thank Alan Chang, Maksym Radziwi\l{}\l{}, Jianghao Zhang, and the anonymous referee for corrections and comments improving the exposition.

\section{Passage from linear to bilinear decoupling}
\label{sect:bilinear}
The main reason allowing for the proof of decoupling inequalities in \cite{MR3374964} is that they can be reduced to multilinear inequalities by an argument introduced in Bourgain--Guth \cite{MR2860188}.
Since the moment curve is one-dimensional, and we are able to treat bilinear, rather than multilinear, inequalities, we managed to use a simpler argument based on a Whitney decomposition of the square $[0,1]^{2}$ around the diagonal.

\begin{definition}
For $\delta \in (0,1/4)$, the \emph{symmetric bilinear decoupling constant $\BilinDec(\delta)$} for the moment curve $\Gamma$ in $\hat{\R}^k$ is the smallest constant such that, for any pair of intervals $I,I' \in \Part{1/4}$ with $\dist(I,I') \geq 1/4$ and any tuple of functions $(f_{J})_{J \in \Part[I]{\delta} \cup \Part[I']{\delta}}$ with $\supp \widehat{f_{J}} \subseteq \calU_{J}$ for all $J$, the following inequality holds:
\begin{equation}
\label{eq:BilinDec}
\int_{\R^k} \abs{f_{I}}^{p_{k}/2} \abs{f_{I'}}^{p_{k}/2} \leq
\BilinDec(\delta)^{p_{k}} \Bigl[ \sum_{J \in \Part[I]{\delta}} \norm{f_J }_{p_{k}}^2 \Bigr]^{p_{k}/4}
\Bigl[ \sum_{J' \in \Part[I']{\delta}} \norm{f_{J'} }_{p_{k}}^2 \Bigr]^{p_{k}/4}.
\end{equation}
\end{definition}

\begin{lemma}[Bilinear reduction] \label{lem:bilinear}
If $\delta=2^{-N}$, then
\begin{equation}
\label{eq:bilinear-reduction}
\mc{D}_{k}(\delta)
\lesssim
\Bigl( 1 + \sum_{n=2}^{N} \BilinDec(2^{-N+n-2})^{2} \Bigr)^{1/2}.
\end{equation}
\end{lemma}

The proof of this lemma relies on affine rescaling, an idea that already underpinned the arguments in \cite{MR2860188}, \cite{MR3374964}, and \cite{MR3548534}.
The idea is based on the observation that, for any interval $I = [a,a+\kappa]$, the affine map $A_I \colon \hat{\R}^k \to \hat{\R}^k$, defined by
\[
\bigl( A_I (\eta_{1},\dotsc,\eta_{k}) \bigr)_{j} := \sum_{j'=0}^{k} \binom{j}{j'} a^{j-j'} \kappa^{j'} \eta_{j'},
\quad 1\leq j\leq k,
\]
where, by convention, $\eta_{0}=1$, satisfies $A_I \Gamma(t) = \Gamma(a + t \kappa)$ for all $t\in\R$, and hence
\[
(DA_I) \partial^i \Gamma(t) = \kappa^i (\partial^i \Gamma)(a+t\kappa) \quad \text{for all $i\geq 1$ and $t\in\R$}.
\]
It follows that, for dyadic intervals $I,J$ with $J \subseteq I \subseteq [0,1]$, we have
\[
A_I^{-1} \calU_{J} = \calU_{J_I},
\]
where $J_I := \kappa^{-1} (J - a)$ if $I = [a,a+\kappa]$. This implies

\begin{lemma}
[Affine rescaling]
\label{lem:rescale}
Let $I \in \Part{2^{-n}}$ for some integer $n \geq 0$.
For any $\delta \in (0,2^{-n})$ and any tuple of functions $(f_J)_{J \in \Part[I]{\delta}}$ with $\supp \widehat{f_J} \subseteq \calU_{J}$ for all $J$, the following inequality holds:
\begin{equation} \label{eq:lin_rescale}
\norm{ f_{I} }_{p_k}
\leq
\Dec_k(2^{n} \delta) \Bigl( \sum_{J \in \Part[I]{\delta}} \norm[\big]{ f_{J} }_{p_k}^{2} \Bigr)^{1/2}.
\end{equation}
Similarly, let $I$, $I' \in \Part{2^{-n}}$ for some integer $n \geq 2$ with $2^n \dist(I,I') \in \Set{1,2}$.
For any $\delta \in (0,2^{-n})$ and any tuple of functions $(f_J)_{J \in \Part[I]{\delta} \cup \Part[I']{\delta}}$ with $\supp \widehat{f_J} \subseteq \calU_{J}$ for all $J$, the following inequality holds:
\begin{equation} \label{eq:bilin_rescale}
\int_{\R^k} \abs{f_{I}}^{p_{k}/2} \abs{f_{I'}}^{p_{k}/2} \leq
\BilinDec(2^{n-2} \delta)^{p_{k}} \Bigl[ \sum_{J \in \Part[I]{\delta}} \norm{f_J }_{p_{k}}^2 \Bigr]^{p_{k}/4}
\Bigl[ \sum_{J' \in \Part[I']{\delta}} \norm{f_{J'} }_{p_{k}}^2 \Bigr]^{p_{k}/4}.
\end{equation}
\end{lemma}

\begin{proof}
To prove \eqref{eq:lin_rescale}, suppose that $I = [a,a+2^{-n}]$.
For $J \in \Part[I]{\delta}$ and $K=J_{I} \in \Part{2^{n}\delta}$, let the function $g_{K}$ be such that $\widehat{f_{J}} \circ A_{I} = \widehat{g_{K}}$.
Applying \eqref{eq:dec-const} to $(g_K)$ in place of $(f_J)$, and changing variables on both sides, we obtain \eqref{eq:lin_rescale}.
A similar argument proves \eqref{eq:bilin_rescale}, which we omit.
\end{proof}

\begin{proof}
[Proof of Lemma~\ref{lem:bilinear}]
Suppose that $\delta=2^{-N}$.
Set $\calW_1:=\emptyset$.
For integers $n\geq 2$, define iteratively
\[
\calW_{n} := \Set[\Big]{ (I_{1},I_{2}) \in \Part{2^{-n}} \given 2^{n}\dist(I_{1},I_{2}) \in \Set{1,2} \text{ and } I_1\times I_2 \not\subset \bigcup_{(I'_1, I'_2)\in \calW_{n-1}}I'_1\times I'_2}.
\]
These are the squares of scale $2^{-n}$ in the Whitney decomposition of the unit square around the diagonal.
Let also
\[
\widetilde{\calW}_{n} := \Set{ (I_{1},I_{2}) \in \Part{2^{-n}} \given \dist(I_{1},I_{2}) = 0}
\]
be the squares of scale $2^{-n}$ that touch the diagonal.
For $N\geq 2$, let
\[
\calW^{N} := \bigcup_{n=2}^{N} \calW_{n} \cup \widetilde{\calW}_{N},
\]
so that the squares $I_{1}\times I_{2}$ with $(I_{1},I_{2})\in\calW^{N}$ form an essentially disjoint (up to boundaries) covering of $[0,1]^{2}$.
Let $(f_{J})_{J \in \Part{\delta}}$ be as in Definition~\ref{def:dec-const} for $\Dec_k(\delta)$. Then
\begin{equation}
\label{eq:1}
\begin{split}
\norm{ \sum_{J \in \Part{\delta}} f_J }_{p_k}
&=
\norm{ \sum_{(I, I') \in \calW^{N}} f_{I}\ov{f_{I'}}}_{p_{k}/2}^{1/2}\leq
\Bigl( \sum_{(I, I') \in \calW^{N}} \norm{ f_{I}\ov{f_{I'}}}_{p_{k}/2} \Bigr)^{1/2}\\
&\leq
\Bigl( \sum_{(I, I') \in \widetilde{\calW}_{N}} \norm{f_{I}}_{p_{k}} \norm{f_{I'}}_{p_{k}}
+ \sum_{n=2}^{N} \sum_{(I, I') \in \calW_{n}} \norm{ f_{I}\ov{f_{I'}}}_{p_{k}/2}\Bigr)^{1/2}.
\end{split}
\end{equation}
We estimate the first term by
\begin{align*}
\sum_{(I, I') \in \widetilde{\calW}_{N}} (\norm{f_{I}}_{p_{k}}^{2} + \norm{f_{I'}}_{p_{k}}^{2}) \leq
6 \sum_{I \in \Part{2^{-N}}} \norm{f_{I}}_{p_{k}}^{2},
\end{align*}
since each $I$ appears at most $6$ times in the pairs $\widetilde{\calW}_{N}$.
In the second term, by affine rescaling \eqref{eq:bilin_rescale}, for every $(I,I')\in\calW_{n}$, we have
\begin{align*}
\norm{ f_{I}\ov{f_{I'}}}_{p_{k}/2}
& \lesssim
\BilinDec(2^{-N+n-2})^{2}
\bigl( \ell^{2}_{J\in\Part[I]{2^{-N}}} \norm{ f_{J}}_{p_{k}} \bigr)
\bigl( \ell^{2}_{J'\in\Part[I']{2^{-N}}} \norm{ f_{J'}}_{p_{k}} \bigr)
\\ &\lesssim
\BilinDec(2^{-N+n-2})^{2}
\Bigl( \bigl( \ell^{2}_{J\in\Part[I]{2^{-N}}} \norm{ f_{J}}_{p_{k}} \bigr)^{2} +
\bigl( \ell^{2}_{J'\in\Part[I']{2^{-N}}} \norm{ f_{J'}}_{p_{k}} \bigr)^{2} \Bigl).
\end{align*}
Since each $I \in \Part{2^{-n}}$ appears at most $8$ times in $\calW_{n}$, it follows that
\begin{align*}
\sum_{(I, I') \in \calW_{n}} \norm{ f_{I}\ov{f_{I'}}}_{p_{k}/2}
&\lesssim
\BilinDec(2^{-N+n-2})^{2}
\bigl( \ell^{2}_{J\in\Part{2^{-N}}} \norm{ f_{J}}_{p_{k}} \bigr)^{2}.
\end{align*}
Inserting these bounds in \eqref{eq:1}, we obtain the desired estimate.
\end{proof}

\section{Lower degree decoupling} 
In this section, we first introduce $k$ new asymmetric bilinear decoupling constants for the moment curve in $\hat{\R}^k$, and relate them to the symmetric ones in Section~\ref{sect:bilinear} (Lemma~\ref{lem:BilinDec_trivial_first_step}).
We then show how these new asymmetric bilinear constants can be bounded efficiently via decoupling for moment curves of degrees $< k$ (Lemma~\ref{lem:BI}).
The key is certain transversality as displayed in Lemma~\ref{lem:transversality}.
Lemma~\ref{lem:BI} will allow us to prove Theorem~\ref{thm:main} in Section~\ref{sect:induction}, by induction on $k$.

\subsection{Asymmetric bilinear decoupling constants}\label{sect:asym_bilinear}
For a dyadic interval $I$, let $\Uspace_{I}$ denote the parallelepiped centered at the origin polar to $\calU_{I}$, that is,
\[
\Uspace_{I} := \Set{ x \in \R^k \given \abs{\langle x, \partial^i \Gamma(c_I) \rangle} \leq \abs{I}^{-i}, 1 \leq i \leq k}.
\]
It is a parallelepiped of dimension $\sim \abs{I}^{-1}\times \abs{I}^{-2}\times \dotsm \times \abs{I}^{-k}$.
Let
\[
\phi_{I}(x) := \abs{\Uspace_{I}}^{-1} \inf\Set{t\geq 1 \given x/t \in \Uspace_{I}}^{-A}
\]
where $A$ is a dimensional constant satisfying $A > k$ and $A \geq \frac{k(k+1)}{2}$. 
$\phi_{I}$ is an $L^{1}$ normalized positive bump function adapted to $\Uspace_{I}$. 
The power $A$ was chosen so that Lemma~\ref{lem:convolution} holds.

\begin{definition}
\label{def:BilinDec_l}
For $l \in \Set{0,\dotsc,k-1}$, $a,b \in [0,1]$ and $\delta \in (0,1)$, the \emph{(asymmetric) bilinear decoupling constant} $\BilinDec_{l, a, b}(\delta)$ for the moment curve $\Gamma$ in $\hat{\R}^k$ is the smallest constant such that, for all pairs of intervals $I \in \Part{\delta^a}$, $I' \in \Part{\delta^b}$ with $\dist(I, I') \ge 1/4$ and all tuples of functions $(f_{J})_{J \in \Part[I]{\delta} \cup \Part[I']{\delta}}$ with $\supp \widehat{f_{J}} \subseteq \calU_{J}$ for all $J$, the following inequality holds:
\begin{equation}
\label{eq:BilinDec_l}
\begin{split}
& \int_{\R^k}\big( \abs{f_{I}}^{p_{l}}*\phi_{I} \big) \big( \abs{f_{I'}}^{p_{k}-p_{l}}* \phi_{I'}\big)\\
&\leq \BilinDec_{l, a, b}(\delta)^{p_{k}}\Bigl[ \sum_{J \in \Part[I]{\delta}} \norm{f_J }_{p_{k}}^2 \Bigr]^{p_{l}/2} \Bigl[ \sum_{J' \in \Part[I']{\delta}} \norm{f_{J'} }_{p_{k}}^2 \Bigr]^{(p_{k}-p_{l})/2}.
\end{split}
\end{equation}
\end{definition}

\begin{remark}
In the case $l=0$, the bilinear decoupling constant $\BilinDec_{0,a,b}(\delta)$ clearly does not depend on $a$, and in fact, by affine rescaling \eqref{eq:lin_rescale}, we have
\begin{equation}
\label{eq:BilinDec_0}
\BilinDec_{0,a,b}(\delta) \sim \Dec_{k}(\delta^{1-b}).
\end{equation}
In order to avoid case distinction in \eqref{eq:Holder} and thereafter, we do not require $a$ in the notation $\BilinDec_{0,a,b}(\delta)$ to be well-defined.
\end{remark}

Our choice of the left hand side of \eqref{eq:BilinDec_l} is partly motivated by the following uncertainty principle.
\begin{lemma}[Uncertainty Principle]
\label{lem:uncertainty}
For $p \in [1,\infty)$ and $J \subset [0, 1]$, we have
\[
\abs{g_{J}}^{p}
\lesssim_{p}
\abs{g_{J}}^{p} * \phi_{J},
\]
for every $g_J$ with $\supp \widehat{g_J} \subseteq C \calU_J$.
\end{lemma}
\begin{proof}
Let $\psi$ be a Schwartz function adapted to $\Uspace_{J}$ such that $\widehat{\psi}\equiv 1$ on $C \calU_J$ and $\int \abs{\psi} \approx 1$.
Then $g_{J} = g_{J}*\psi$, so
\begin{align*}
\abs{g_{J}}^{p}(x)
&\leq
\Bigl( \int \abs{g_{J}(x-z)}^{p} \abs{\psi(z)} \dif z \Bigr)
\Bigl( \int \abs{\psi(z)} \dif z \Bigr)^{p/p'}
\numberthis\label{eq:8}
\\ &\lesssim
(\abs{g_{J}}^{p} * \abs{\psi})(x)
\lesssim
(\abs{g_{J}}^{p} * \phi_{J})(x).
\qedhere
\end{align*}
\end{proof}
The first application of Lemma~\ref{lem:uncertainty} is that the symmetric bilinear decoupling constants \eqref{eq:BilinDec} can be bounded (rather crudely) by the asymmetric ones \eqref{eq:BilinDec_l}.

\begin{lemma}
\label{lem:BilinDec_trivial_first_step}
For every $l \in \Set{0,\dotsc,k-1}$, $a,b \in [0,1]$ and $\delta \in (0,1/4)$, we have
\begin{equation}
\label{eq:BilinDec_trivial_first_step}
\BilinDec(\delta)
\lesssim
\delta^{-a p_{l}/p_{k}} \delta^{-b (p_{k}-p_{l})/p_{k}} \BilinDec_{l, a, b}(\delta).
\end{equation}
\end{lemma}
\begin{proof}
Let $I,I' \in \Part{1/4}$ with $\dist(I,I') \geq 1/4$.
Let $(f_{K})_{K \in \Part[I]{\delta} \cup \Part[I']{\delta}}$ be a tuple of functions with $\supp \widehat{f_{K}} \subseteq \calU_{K}$ for all $K$.
By H\"older's inequality, we have
\begin{equation} \label{eq:Holdersym}
\int_{\R^k} \abs{f_{I}}^{p_{k}/2} \abs{f_{I'}}^{p_{k}/2}
\leq
\Bigl( \int_{\R^k} \abs{f_{I}}^{p_{l}} \abs{f_{I'}}^{p_{k}-p_{l}} \Bigr)^{1/2}
\Bigl( \int_{\R^k} \abs{f_{I}}^{p_{k}-p_{l}} \abs{f_{I'}}^{p_{l}} \Bigr)^{1/2}.
\end{equation}
By symmetry, it suffices to estimate the first bracket.
Assume that $l\neq 0$; the case $l=0$ is similar, but easier, since the term with power $p_{l}$ disappears.
We have
\begin{align*}
\int_{\R^k} \abs{f_{I}}^{p_{l}} \abs{f_{I'}}^{p_{k}-p_{l}}
&\leq
\int_{\R^k} \bigl(\sum_{J\in\Part[I]{\delta^{a}}} \abs{f_{J}} \bigr)^{p_{l}} \bigl(\sum_{J'\in\Part[I']{\delta^{b}}} \abs{f_{J'}} \bigr)^{p_{k}-p_{l}}
\\ &\leq
\abs{\Part[I]{\delta^{a}}}^{p_{l}-1} \abs{\Part[I']{\delta^{b}}}^{p_{k}-p_{l}-1}
\sum_{J\in\Part[I]{\delta^{a}}} \sum_{J'\in\Part[I']{\delta^{b}}}
\int_{\R^k} \abs{f_{J}}^{p_{l}} \abs{f_{J'}}^{p_{k}-p_{l}}.
\end{align*}
By Lemma~\ref{lem:uncertainty} and Definition~\ref{eq:BilinDec_l}, we have
\begin{align*}
\int_{\R^k} \abs{f_{J}}^{p_{l}} \abs{f_{J'}}^{p_{k}-p_{l}}
&\lesssim
\int_{\R^k} (\abs{f_{J}}^{p_{l}} * \phi_{J}) (\abs{f_{J'}}^{p_{k}-p_{l}} * \phi_{J'})
\\ &\leq
\BilinDec_{l, a, b}(\delta)^{p_{k}}
\Bigl[ \ell^{2}_{K \in \Part[J]{\delta}} \norm{f_{K} }_{p_{k}} \Bigr]^{p_{l}}
\Bigl[ \ell^{2}_{K' \in \Part[J']{\delta}} \norm{f_{K'} }_{p_{k}} \Bigr]^{p_{k}-p_{l}}.
\end{align*}
Inserting this into the previous display, and using $\ell^2 \hookrightarrow \ell^{p_l}, \ell^{p_k-p_l}$, we obtain
\begin{multline*}
\int_{\R^k} \abs{f_{I}}^{p_{l}} \abs{f_{I'}}^{p_{k}-p_{l}}
\lesssim
\delta^{-a(p_{l}-1)} \delta^{-b(p_{k}-p_{l}-1)} \BilinDec_{l, a, b}(\delta)^{p_{k}}
\\ \cdot
\Bigl[ \ell^{2}_{K \in \Part[I]{\delta}} \norm{f_{K} }_{p_{k}} \Bigr]^{p_{l}}
\Bigl[ \ell^{2}_{K' \in \Part[I']{\delta}} \norm{f_{K'} }_{p_{k}} \Bigr]^{p_{k}-p_{l}}.
\end{multline*}
Together with a similar estimate for the second factor in \eqref{eq:Holdersym}, we obtain the desired estimate.
\end{proof}

\subsection{Transversality}
Let $V^{(l)}(\xi)$ denote the $l$-th order tangent space to the moment curve $\Gamma$ at the point $\xi$, that is,
\[
V^{(l)}(\xi) := \lin (\partial^{1}\Gamma(\xi), \dotsc, \partial^{l}\Gamma(\xi) ).
\]
The main geometric observation that makes our inductive argument work is that the spaces $V^{(l)}(\xi_{1})$ and $V^{(k-l)}(\xi_{2})$ are transverse for any $l \in \Set{1,\dotsc,k-1}$, as long as $\xi_{1}\neq \xi_{2}$.
This transversality is made quantitative in the following result. It follows from the generalized Vandermonde determinant formula in \cite[Equation (14)]{MR729034}; we include a proof for completeness.
\begin{lemma}
\label{lem:transversality}
For any integers $0 \leq l \leq k$ and any $\xi_{1},\xi_{2} \in \R$, we have
\begin{equation}
\label{eq:transversality}
\abs[\Big]{ \partial^{1}\Gamma(\xi_{1}) \wedge \dotsb \wedge \partial^{l}\Gamma(\xi_{1})
\wedge
\partial^{1}\Gamma(\xi_{2}) \wedge \dotsb \wedge \partial^{k-l}\Gamma(\xi_{2}) }
\gtrsim_{k,l} \abs{\xi_{1}-\xi_{2}}^{l(k-l)}.
\end{equation}
\end{lemma}
\begin{proof}
We Taylor expand $\Gamma(\xi_2)$ around $\xi_1$: for $1 \leq i \leq k-l$,
\[
\partial^i \Gamma(\xi_2) = \sum_{j=i}^k \frac{1}{(j-i)!} \partial^j \Gamma(\xi_1) (\xi_2-\xi_1)^{j-i}.
\]
We plug this back to the left hand side of \eqref{eq:transversality}, and obtain an $k-l$ fold sum. If $\partial^{j_i} \Gamma$ is chosen for the $i$-th summand, then $(j_1,\dotsc,j_{k-l})$ has to be a permutation of $(l+1,\dotsc,k)$ in order for the term to be non-zero, in which case the power of $\xi_2-\xi_1$ is
\[
\sum_{i=1}^{k-l} (j_i-i) = ((l+1)+\dotsc+k)-(1+\dotsc+(k-l)) = l(k-l).
\]
Thus the left hand side of \eqref{eq:transversality} is equal to
\[
c_{k,l} \abs[\Big]{ \partial^{1}\Gamma(\xi_{1}) \wedge \dotsb \wedge \partial^{k}\Gamma(\xi_{1})} \abs{\xi_2-\xi_1}^{l(k-l)}
\]
for some constant $c_{k,l} \geq 0$. Setting $\xi_1 = 0$ and $\xi_2 = 1$ shows that $c_{k,l} > 0$; indeed then
the left hand side of \eqref{eq:transversality} is $\binom{k}{l} (\prod_{i=1}^l i!) (\prod_{j=1}^{k-l} j!)$, as can be seen by column operations and the classical Vandermonde determinant formula.
See also \cite{MR3994585}, \cite{arxiv:1811.02207} for similar calculations.
\end{proof}

\subsection{Decoupling for curves with torsion}
It is an observation going back to \cite[Proposition 2.1]{MR2288738} that decoupling inequalities for model manifolds self-improve to similar decoupling inequalities for similarly curved manifolds.
We need the following version of Theorem~\ref{thm:main} for more general curves with torsion, which is proved by the argument given in \cite[Section 7]{MR3374964}.

Suppose $l \in \N$ and $\gamma : [0,1] \to \hat{\R}^{l}$ is a curve such that
\begin{equation}
\label{eq:torsion}
\norm{\gamma}_{C^{l+1}} \lesssim 1
\quad\text{and}\quad
\abs{ \partial^{1}\gamma(\xi) \wedge \dotsm \wedge \partial^{l}\gamma(\xi) } \gtrsim 1.
\end{equation}
For dyadic intervals $J$, let $\calU_{J,\gamma}$ be the parallelepiped of dimensions $\abs{J}^{1} \times \dotsm \times \abs{J}^{l}$ whose center is $\gamma(c_J)$ and sides are parallel to $\partial^{1}\gamma(c_{J}),\dotsc,\partial^{l}\gamma(c_{J})$, and let $\Uspace_{J,\gamma}$ be polar to $\calU_{J,\gamma}$.
\begin{lemma}
\label{lem:torsion}
Suppose that Theorem~\ref{thm:main} is known with $k$ replaced by $l$.
Let $\gamma : [0,1] \to \hat{\R}^{l}$ be a curve satisfying \eqref{eq:torsion}.
Then for any $\epsilon, C >0$, any $\delta \in (0,1)$, and any tuple of functions $(f_{J})_{J \in \Part{\delta}}$ with $\supp \widehat{f_{J}} \subseteq C \calU_{J,\gamma}$ for all $J$, the following inequality holds:
\begin{equation}
\label{eq:dec-with-torsion}
\norm[\Big]{ \sum_{J \in \Part{\delta}} f_{J} }_{L^{p_{l}}(\R^{l})}
\lesssim_{\epsilon, C} \delta^{-\epsilon}
\Bigl( \sum_{J \in \Part{\delta}} \norm[\big]{ f_{J} }_{L^{p_{l}}(\R^{l})}^{2} \Bigr)^{1/2}.
\end{equation}
\end{lemma}
\begin{proof}
Let $(f_{J})_{J \in \Part{\delta}}$ be a tuple of functions with $\supp \widehat{f_{J}} \subseteq C \calU_{J,\gamma}$ for all $J$.
It suffices to show that, for every $\kappa > \delta^{l/(l+1)}$ and $I \in \Part{\kappa}$, we have
\begin{equation}
\label{eq:torsion:induction-step}
\norm[\big]{ f_{I} }_{L^{p_{l}}(\R^{l})}
\lesssim_{\epsilon} \kappa^{-\epsilon}
\ell^{2}_{I' \in \Part[I]{\kappa^{(l+1)/l}}} \norm[\big]{ f_{I'} }_{L^{p_{l}}(\R^{l})}
\end{equation}
where we abbreviated $f_{I'} := \sum_{J\in\Part[I']{\delta}} f_{J}$ for $I' \in \Part[I]{\kappa^{(l+1)/l}}$ and similarly for $f_{I}$.

Indeed, if \eqref{eq:torsion:induction-step} is known, then we can use a trivial decoupling inequality to reduce to the case that $f_{J} \neq 0$ only if $J \subseteq I$ for some $I \in \Part{\delta^{(l/(l+1))^{M}}}$ for a large integer $M$, and then apply \eqref{eq:torsion:induction-step} $M$ times.
This will give \eqref{eq:dec-with-torsion} with power, say, $(l/(l+1))^{M} (l+1) + l \epsilon$ in place of $\epsilon$.
Since $M$ is arbitrary, this concludes the proof.

To see that \eqref{eq:torsion:induction-step} holds, observe that, on the interval $I$, we have
\[
\gamma(\xi) = \underbrace{\gamma(c_{I}) + \partial^{1} \gamma(c_{I}) \cdot (\xi-c_{I}) + \dotsb + \frac{\partial^{l} \gamma(c_{I})}{l!} \cdot (\xi-c_{I})^{l}} + O(\kappa^{l+1}).
\]
By \eqref{eq:torsion}, the marked part of the above expression is, up to a uniformly non-singular affine transformation, a moment curve of degree $l$.
For every $I' \in \Part[I]{\kappa^{(l+1)/l}}$, we have $\supp \widehat{f_{I'}} \subseteq C \calU_{I',\gamma}$, and the parallelepiped $\calU_{I',\gamma}$ is contained in a similar parallelepiped associated to this moment curve, since the shortest side of $\calU_{I',\gamma}$ is $(\kappa^{(l+1)/l})^l \gtrsim O(\kappa^{l+1})$.
Hence, the claim \eqref{eq:torsion:induction-step} follows from a rescaled version of Theorem~\ref{thm:main}; see \eqref{eq:lin_rescale} and its proof.
\end{proof}

\begin{corollary}
\label{cor:torsion-local}
In the situation of Lemma~\ref{lem:torsion}, for any $A'>0$ and for every ball $B \subset \R^{l}$ of radius $\delta^{-l}$, we have
\[
\fint_{B} \abs[\Big]{ \sum_{J \in \Part{\delta}}f_{J}}^{p_{l}}
\lesssim_{\epsilon,C,A'} \delta^{-\epsilon}
\Bigl( \ell^{2}_{J \in \Part{\delta}} \norm{f_{J}}_{L^{p_{l}}(\Phi_{B})} \Bigr)^{p_{l}},
\]
where $\fint_{B} := \abs{B}^{-1} \int_{B}$ denotes the average integral and
\[
\Phi_{B}(x) := \abs{B}^{-1} (1+\delta^{l}\dist(x,B))^{-A'}
\]
is an $L^{1}$ normalized bump function adapted to $B$.
\end{corollary}
\begin{proof}
Apply Lemma~\ref{lem:torsion} to functions $f_{J}\psi_{B}$, where $\psi_{B}$ is a Schwartz function such that $\abs{\psi_{B}} \sim 1$ on $B$ and $\supp \widehat{\psi_{B}} \subseteq B(0,\delta^{l})$.
\end{proof}
We will use Corollary~\ref{cor:torsion-local} with $A' := A + k - 2$, where $A$ is the exponent occurring in the definition of $\phi_I$.
The choice of the exponent $A'$ ensures that Lemma~\ref{lem:convolution-H} holds.

\subsection{Using the lower degree inductive hypothesis}
The following two key lemmas should be compared to Lemma 7.1 of \cite{MR3938716}, which plays a similarly key role in nested efficient congruencing.
The results below improve upon those in \cite{arxiv:1906.07989} by incorporating sharp canonical scale decoupling inequalities of all degrees $l<k$, whereas in \cite{arxiv:1906.07989} small ball decoupling, which is not yet known for higher degrees, was used in the case $l=2$.

\begin{lemma}[Lower degree decoupling] \label{lem:b}
Let $l \in \Set{1,\dotsc,k-1}$ and assume that Theorem~\ref{thm:main} is known with $k$ replaced by $l$.
Let $\delta \in (0,1)$ and $(f_K)_{K \in \Part{\delta}}$ be a tuple of functions so that $\supp \widehat{f_K} \subset \calU_K$ for every $K$.
If $0 \leq a \leq (k-l+1)b/l$, then, for any pair of frequency intervals $I \in \Part{\delta^a}$, $I' \in \Part{\delta^b}$ with $\dist(I,I') \geq 1/4$, we have
\begin{equation}
\label{eq:fiber-dec}
\begin{split}
& \int_{\R^k} \big( \abs{f_{I}}^{p_{l}}*\phi_{I} \big) \big( \abs{f_{I'}}^{p_{k}-p_{l}}* \phi_{I'}\big)\\
& \lesssim_{\epsilon} \delta^{-b\epsilon} \Bigl( \sum_{J\in \Part[I]{\delta^{(k-l+1)b/l}}} \Bigl( \int_{\R^k} \big( \abs{f_{J}}^{p_{l}}*\phi_{J} \big) \big( \abs{f_{I'}}^{p_{k}-p_{l}}* \phi_{I'}\big) \Bigr)^{2/p_{l}} \Bigr)^{p_{l}/2}.
\end{split}
\end{equation}
\end{lemma}
The above lemma motivates our carefully chosen definition of asymmetric bilinear decoupling constants. It immediately implies the following result.
\begin{lemma}\label{lem:BI}
Let $l \in \Set{1,\dotsc,k-1}$ and assume that Theorem~\ref{thm:main} is known with $k$ replaced by $l$.
Then, for any $0 \leq a \leq \frac{k-l+1}{l}b$, $\epsilon>0$, and $\delta \in (0,1)$, we have
\[
\BilinDec_{l,a,b}(\delta)
\lesssim_{\epsilon} \delta^{-b\epsilon}
\BilinDec_{l, \frac{k-l+1}{l}b, b}(\delta).
\]
\end{lemma}

\begin{proof}[Proof of Lemma~\ref{lem:b}]
Denote $b' := (k-l+1)b/l$.
Fix $\xi'\in I'$ and let $\hat{H} := \hat{\R}^{k}/V^{k-l}(\xi')$ be the quotient space.
Let $P : \hat{\R}^{k} \to \hat{H}$ be the projection onto $\hat{H}$.
For every $\xi \in I$, it follows from Lemma~\ref{lem:transversality} that
\[
\abs{ \partial^{1}(P\circ\Gamma)(\xi) \wedge \dotsm \wedge \partial^{l}(P\circ\Gamma)(\xi) } \gtrsim 1.
\]
Moreover, $P(\calU_{J}) \subseteq C \calU_{J, P \circ \Gamma}$.
Let $H$ be the orthogonal complement of $V^{k-l}(\xi')$ in $\R^{k}$, so that $\hat{H}$ is its Pontryagin dual.
Since the Fourier support of the restriction $f_{J} |_{H+z}$ to almost every translated copy of $H$ is contained in the projection of the Fourier support of $f_{J}$ onto $\hat{H}$, we will be able to apply Corollary~\ref{cor:torsion-local} on almost every translate $H+z$.

To be more precise, by Fubini's theorem, we write
\begin{equation}
\label{eq:2}
\int_{\R^k} \big( \abs{f_{I}}^{p_{l}}*\phi_{I} \big) \big( \abs{f_{I'}}^{p_{k}-p_{l}}* \phi_{I'}\big)
=
\int_{z \in \R^{k}} \fint_{B_H(z,\delta^{-b'l})} \big( \abs{f_{I}}^{p_{l}}*\phi_{I} \big) \big( \abs{f_{I'}}^{p_{k}-p_{l}}* \phi_{I'}\big),
\end{equation}
where $B_H(z,\delta^{-b' l})$ is the $l$-dimensional ball with radius $\delta^{-b'l}$ centered at $z$ inside the affine subspace $H+z$.
Since $B_H(0,\delta^{-b'l}) = B_H(0,\delta^{-(k-l+1)b}) \subseteq C \Uspace_{I'}$, we have
\[
\sup_{x \in B_H(z,\delta^{-b'l})} \big( \abs{f_{I'}}^{p_{k}-p_{l}}* \phi_{I'}\big)(x)
\lesssim
\big( \abs{f_{I'}}^{p_{k}-p_{l}}* \phi_{I'}\big)(z).
\]
Applying this estimate in \eqref{eq:2}, we are led to bound
\begin{align*}
\fint_{B_H(z,\delta^{-b' l})} \big( \abs{f_{I}}^{p_{l}}*\phi_{I} \big)
&=
\abs{f_{I}}^{p_{l}}*\phi_{I}*_H \frac{\one_{B_H(0,\delta^{-b'l})}}{\abs{B_H(0, \delta^{-b'l})}}(z)
=
\int_{z'} \phi_{I}(z-z') \fint_{B_H(z',\delta^{-b'l})} \abs{f_{I}}^{p_{l}}
\end{align*}
where $*_H$ denotes convolution on the subspace $H$.
By Corollary~\ref{cor:torsion-local} with $\delta^{b'}$ in place of $\delta$ applied to the curve $\gamma = P \circ \Gamma$, the above is further bounded by
\begin{align*}
&\lesssim_{\epsilon}
\delta^{-b\epsilon} \int_{z'} \phi_{I}(z-z') \Bigl( \ell^{2}_{J \in \Part[I]{\delta^{b'}}} \norm{f_{J}}_{L^{p_{l}}(\Phi_{B_H(z',\delta^{-b'l})})} \Bigr)^{p_{l}}.
\end{align*}
Hence, the $p_l$-th root of \eqref{eq:2} can be bounded by
\begin{align*}
\eqref{eq:2}^{1/p_{l}}
&\lesssim_{\epsilon}
\delta^{-b\epsilon}
\Bigl( \int_{z,z'\in\R^{k}} \big( \abs{f_{I'}}^{p_{k}-p_{l}}* \phi_{I'}\big)(z) \phi_{I}(z-z') \Bigl( \ell^{2}_{J \in \Part[I]{\delta^{b'}}} \norm{f_{J}}_{L^{p_{l}}(z'+H,\Phi_{B_H(z',\delta^{-b'l})})} \Bigr)^{p_{l}} \Bigr)^{1/p_{l}}
\\ &\leq
\delta^{-b\epsilon}
\ell^{2}_{J \in \Part[I]{\delta^{b'}}} \Bigl( \int_{z,z'\in\R^{k}} \big( \abs{f_{I'}}^{p_{k}-p_{l}}* \phi_{I'}\big)(z) \phi_{I}(z-z') \norm{f_{J}}_{L^{p_{l}}(\Phi_{B_H(z',\delta^{-b'l})})}^{p_{l}} \Bigr)^{1/p_{l}},
\end{align*}
where we used Minkowski's inequality in the form $L^{p_{l}}\ell^{2} \leq \ell^{2}L^{p_{l}}$.
The double integral inside the brackets can be written as
\begin{align*}
\MoveEqLeft
\int_{\R^{k}} \bigl( \abs{f_{I'}}^{p_{k}-p_{l}} * \phi_{I'} \bigr) \bigl( \phi_{I} * \abs{f_{J}}^{p_{l}} *_{H} \Phi_{B_H(0,\delta^{-b'l})} \bigr)
\\ &=
\int_{\R^{k}} \bigl( \abs{f_{I'}}^{p_{k}-p_{l}} * \phi_{I'} *_{H} \Phi_{B_H(0,\delta^{-b'l})} \bigr) \bigl( \abs{f_{J}}^{p_{l}} * \phi_{I} \bigr)
\\ &\lesssim
\int_{\R^{k}} \bigl( \abs{f_{I'}}^{p_{k}-p_{l}} * \phi_{I'} \bigr) \bigl( \abs{f_{J}}^{p_{l}} * \phi_{I} \bigr),
\end{align*}
where we used $b'l = (k-l+1)b$ and Lemma~\ref{lem:convolution-H}.
This is in turn
\[
\lesssim
\int_{\R^k} \big( \abs{f_{J}}^{p_{l}}*\phi_{J} \big) \big( \abs{f_{I'}}^{p_{k}-p_{l}}* \phi_{I'}\big),
\]
because $ \abs{f_{J}}^{p_{l}} * \phi_{I} \lesssim \abs{f_{J}}^{p_{l}}*\phi_{J}*\phi_{I}$ by Lemma~\ref{lem:uncertainty}, which is $\lesssim \abs{f_{J}}^{p_{l}}*\phi_{J}$ by Lemma~\ref{lem:convolution}.
\end{proof}

\section{Bootstrap and Iteration}
\label{sect:induction}
In this section, we will prove Theorem~\ref{thm:main}, using Lemma~\ref{lem:BI}.
\begin{lemma}[H\"{o}lder] \label{lem:Holder}
For $l \in \Set{1, \dotsc, k-1}$, if $a, b \in (0,1)$ and $\delta \in (0,1)$, then
\begin{equation}
\label{eq:Holder}
\BilinDec_{l,a,b}(\delta)
\leq
\BilinDec_{k-l,b,a}(\delta)^{\frac{1}{k-l+1}} \BilinDec_{l-1,a,b}(\delta)^{\frac{k-l}{k-l+1}}.
\end{equation}
\end{lemma}
\begin{proof}
For $1 \leq l < k$, the points $(p_l,p_k-p_l)$, $(p_k-p_{k-l},p_{k-l})$ and $(p_{l-1},p_k-p_{l-1})$ are collinear, since their coordinates sum to $p_{k}$.
Hence, there exists $\theta_l \in \R$ such that
\begin{equation} \label{}
(p_l,p_k-p_l) = \theta_l (p_k-p_{k-l},p_{k-l}) + (1-\theta_l) (p_{l-1},p_k-p_{l-1}).
\end{equation}
Substituting $p_l = l(l+1)$ yields $\theta_l = 1/(k-l+1)$.
Let $f_I$, $f_{I'}$ be as in Definition~\ref{def:BilinDec_l} for $\BilinDec_{l,a,b}(\delta)$. By H\"older's inequality, we obtain
\begin{align*}
LHS\eqref{eq:BilinDec_l}
&\leq
\int_{\R^k} \big( \abs{f_{I}}^{p_{k}-p_{k-l}}*\phi_{I} \big)^{\theta_{l}} \big( \abs{f_{I}}^{p_{l-1}}*\phi_{I} \big)^{1-\theta_{l}}
\big( \abs{f_{I'}}^{p_{k-l}}* \phi_{I'}\big)^{\theta_{l}} \big( \abs{f_{I'}}^{p_{k}-p_{l-1}}* \phi_{I'}\big)^{1-\theta_{l}}
\\ &\leq
\Bigl( \int_{\R^k} \big( \abs{f_{I}}^{p_{k}-p_{k-l}}*\phi_{I} \big) \big( \abs{f_{I'}}^{p_{k-l}}* \phi_{I'}\big) \Bigr)^{\theta_{l}}
\Bigl( \int_{\R^k} \big( \abs{f_{I}}^{p_{l-1}}*\phi_{I} \big)
\big( \abs{f_{I'}}^{p_{k}-p_{l-1}}* \phi_{I'}\big) \Bigr)^{1-\theta_{l}}.
\end{align*}
The claim \eqref{eq:Holder} then follows from the definitions of $\BilinDec_{k-l,b,a}(\delta)$ and $\BilinDec_{l-1,a,b}(\delta)$.
\end{proof}

\begin{lemma}
\label{lem:induct_key}
Let $l \in \Set{1,\dotsc,k-1}$ and assume that Theorem~\ref{thm:main} is known with $k$ replaced by $l$.
Let $\epsilon>0$.
Then, for every $b \in [0,1]$ such that $b \leq \frac{l(k-l)}{(l+1)(k-l+1)}$ and, if $l\neq 1$, in addition $b \leq \frac{l-1}{k-l+2}$, we have
\[
\BilinDec_{l,\frac{k-l+1}{l}b,b}(\delta) \lesssim_{\epsilon} \delta^{-b\epsilon}
\BilinDec_{k-l,\frac{l+1}{l} \frac{k-l+1}{k-l}b,\frac{k-l+1}{l}b}(\delta)^{\frac{1}{k-l+1}} \BilinDec_{l-1,\frac{k-l+2}{l-1}b,b}(\delta)^{\frac{k-l}{k-l+1}}.
\]
\end{lemma}

\begin{proof}
Just apply Lemma~\ref{lem:Holder}:
\[
\BilinDec_{l,\frac{k-l+1}{l}b,b}(\delta) \leq \BilinDec_{k-l,b,\frac{k-l+1}{l}b}(\delta)^{\frac{1}{k-l+1}} \BilinDec_{l-1,\frac{k-l+1}{l}b,b}(\delta)^{\frac{k-l}{k-l+1}}
\]
and then estimate the two factors on the right hand side using Lemma~\ref{lem:BI}.
In the first factor, we can apply Lemma~\ref{lem:BI} because
\[
b \leq \frac{l+1}{k-l} \frac{k-l+1}{l}b.
\]
If $2 \leq l \leq k-1$, then we can apply Lemma~\ref{lem:BI} in the second factor because
\[
\frac{k-l+1}{l}b \leq \frac{k-l+2}{l-1}b.
\]
If $l = 1$, the we do not have to do anything in the second factor, since $\BilinDec_{0,a,b}(\delta)$ does not depend on $a$.
\end{proof}

\begin{proof}[Proof of Theorem~\ref{thm:main}]
By induction on $k$.
The case $k=1$ is a direct consequence of Plancherel's theorem.
Fix $k \geq 2$ and assume that Theorem~\ref{thm:main} is already known with $k$ replaced by $l$ for any $l \in \Set{1,\dotsc,k-1}$.

Let $\eta$ be the infimum of all $\epsilon$ for which the decoupling inequality \eqref{eq:main} holds.
For $l \in \Set{0,\dotsc,k-1}$ and $0<b\ll 1$, let $A_{l}(b)$ be the infimum of all exponents $A$ such that we have
\[
\BilinDec_{l,\frac{k-l+1}{l}b,b}(\delta)
\lesssim \delta^{-A}.
\]
By \eqref{eq:BilinDec_0}, we have
\begin{equation}
\label{eq:6}
A_{0}(b) = \eta (1-b).
\end{equation}
The main recursive estimate for the exponents $A_{l}(b)$ is given by Lemma~\ref{lem:induct_key}, which implies that, for every $l \in \Set{1,\dotsc,k-1}$ and sufficiently small $b$, we have
\begin{equation}
\label{eq:3}
A_{l}(b) \leq \frac{1}{k-l+1} A_{k-l}(\frac{k-l+1}{l}b) + \frac{k-l}{k-l+1} A_{l-1}(b).
\end{equation}
We extract the information on the asymptotic behaviour of bilinear decoupling exponents $A_{l}(b)$ from the functional inequality \eqref{eq:3} by introducing the quantities
\[
A_{l} := \liminf_{b\to 0} \frac{\eta - A_{l}(b)}{b} \in \R \cup \Set{\pm\infty}.
\]
By \eqref{eq:6}, we have $A_{0}=\eta$.
Moreover, from \eqref{eq:3}, it follows that
\begin{equation}
\label{eq:7}
A_{l} \geq \frac{1}{l} A_{k-l} + \frac{k-l}{k-l+1} A_{l-1},
\quad
1 \leq l \leq k-1.
\end{equation}
In order to solve this linear system of inequalities for $\eta = A_{0}$, we need to know that the quantities $A_{l}$ are finite, so that we can perform algebraic operations.
The finiteness of these quantities is a manifestation of the equivalence between linear and bilinear decoupling inequalities.

By H\"older's inequality, similarly as in \eqref{eq:8}, for any $l\in\Set{1,\dotsc,k-1}$, $I \in \Part{\delta^{\frac{k-l+1}{l}b}}$, and $I'\in\Part{\delta^b}$, if $\supp \widehat{f_{I}} \subset C \calU_{I}$ and $\supp \widehat{f_{I'}} \subset C \calU_{I'}$, we have
\begin{align*}
\int_{\R^k}\big( \abs{f_{I}}^{p_{l}}*\phi_{I} \big) \big( \abs{f_{I'}}^{p_{k}-p_{l}}* \phi_{I'}\big)
&\leq
\Bigl( \int_{\R^k}\big( \abs{f_{I}}^{p_{l}}*\phi_{I} \big)^{\frac{p_{k}}{p_{l}}} \Bigr)^{\frac{p_{l}}{p_{k}}}
\Bigl( \int_{\R^k} \big( \abs{f_{I'}}^{p_{k}-p_{l}}* \phi_{I'}\big)^{\frac{p_{k}}{p_{k}-p_{l}}} \Bigr)^{\frac{p_{k}-p_{l}}{p_{k}}}
\\ &\lesssim
\Bigl( \int_{\R^k} \abs{f_{I}}^{p_{k}}*\phi_{I} \Bigr)^{\frac{p_{l}}{p_{k}}}
\Bigl( \int_{\R^k} \abs{f_{I'}}^{p_{k}}*\phi_{I'} \Bigr)^{\frac{p_{k}-p_{l}}{p_{k}}}
\\ &\lesssim
\norm{f_{I}}_{p_{k}}^{p_{l}} \norm{f_{I'}}_{p_{k}}^{p_{k}-p_{l}}.
\end{align*}
It follows that, for $l\in \Set{1,\dotsc,k-1}$, we have
\[
\BilinDec_{l,\frac{k-l+1}{l}b,b}(\delta)
\lesssim
\Dec_{k}(\delta^{1-\frac{k-l+1}{l}b})^{p_{l}/p_{k}}
\Dec_{k}(\delta^{1-b})^{(p_{k}-p_{l})/p_{k}}.
\]
Hence,
\begin{equation}
\label{eq:5}
A_{l}(b) \leq
\eta (1-\frac{k-l+1}{l}b) \frac{p_{l}}{p_{k}} +
\eta (1-b) \frac{p_{k}-p_{l}}{p_{k}}
=
\eta - \eta b (\frac{k-l+1}{l} \frac{p_{l}}{p_{k}} + \frac{p_{k}-p_{l}}{p_{k}}).
\end{equation}
Using Lemma~\ref{lem:bilinear} and Lemma~\ref{lem:BilinDec_trivial_first_step}, we see that for every $l\in \Set{1,\dotsc,k-1}$ and every $b \in [0,1]$ with $b \leq \frac{l}{k-l+1}$, we have
\begin{equation}
\label{eq:4}
\eta \leq Cb + A_{l}(b).
\end{equation}
The estimates \eqref{eq:5} and \eqref{eq:4} imply $\eta \lesssim A_{l} \leq C$ for $l\in\Set{1,\dotsc,k-1}$, and in particular that $A_{l}$ are finite numbers.

Summing the inequalities \eqref{eq:7} over $l=1,\dotsc,k-1$, we observe that $A_{1},\dotsc,A_{k-1}$ cancel out, and we are left with
\[
0 \geq \frac{k-1}{k} A_{0} = \frac{k-1}{k} \eta.
\]
This shows that the decoupling exponent is $\eta=0$.
\end{proof}

\begin{remark}
The fact that all $A_{l}$ with $1\leq l \leq k-1$ cancel out when we sum the inequalities \eqref{eq:7} can be more abstractly stated by saying that $(1,\dotsc,1)$ is a left eigenvector of the $(k-1)\times (k-1)$ coefficient matrix
\[
\begin{pmatrix}
0 & 0 & 0 & \dots & 0 & 0 & 1 \\
\frac{k-2}{k-1} & 0 & 0 & \dots & 0 & \frac{1}{2} & 0 \\
0 & \frac{k-3}{k-2} & 0 & \dots & \frac{1}{3} & 0 & 0 \\
& & \ddots & \iddots & & &\\
& & \iddots & \ddots & & &\\
0 & \frac{1}{k-2} & 0 & \dots & \frac{2}{3} & 0 & 0 \\
\frac{1}{k-1} & 0 & 0 & \dots & 0 & \frac{1}{2} & 0
\end{pmatrix},
\]
where the entry at the position $(l,l')$ is the coefficient of $A_{l'}$ on the right-hand side of the $l$-th inequality in \eqref{eq:7}.
We refer to \cite{arxiv:1512.03272} and \cite[Section 3.6]{arxiv:1811.02207} for a discussion of the role of such (Perron--Frobenius) eigenvectors in iterative procedures that are used to prove decoupling inequalities.
\end{remark}

\appendix
\section{Estimates for convolutions of bump functions}
The published version of this article used bump functions $\phi_{I}$  and $\Phi_{B}$ (defined in Section \ref{sect:asym_bilinear} and Corollary \ref{cor:torsion-local}, respectively) 
with exponents $A=A'=10k$, for which the estimates below do not quite work out.
As a remedy, we choose $A$ such that $A > k$ and $A \geq \frac{k(k + 1)}{2}$ and $A' = A + k - 2$.
We present detailed arguments with the new choice of these exponents.
\begin{lemma}
\label{lem:convolution}
For any dyadic intervals $J \subseteq I \subseteq [0,1]$, we have
\begin{equation}
\label{eq:11}
(\phi_I * \phi_J)(x) \lesssim \phi_{J}(x).
\end{equation}
\end{lemma}
\begin{proof}
One can rescale so that $I = [0,1]$. 
By rotating the coordinate system so that the axes are parallel to the sides of the smallest rectangle containing $\Uspace_{J}$, and writing $\delta := |J| \in (0,1]$, $\tilde{\phi}(x) := \max\{1,|x_1|,\dots,|x_k|\}^{-A}$ and $\tilde{\phi}_{\delta}(x) := \delta^{\frac{k(k+1)}{2}} \tilde{\phi}(\delta x_1, \dots, \delta^k x_k)$, the desired pointwise estimate can be written as 
\begin{equation} \label{eq:phi_tilde_bdd}
\int_{\R^k} \tilde{\phi}(x-y) \tilde{\phi}_{\delta}(y) dy \lesssim \tilde{\phi}_{\delta}(x)
\end{equation}
for all $x \in \R^k$. 
This estimate can be established by spliting the integral into two parts, noting that
\[
\int_{\tilde{\phi}_{\delta}(y) \leq 2^{A} \tilde{\phi}_{\delta}(x)} \tilde{\phi}(x-y) \tilde{\phi}_{\delta}(y) dy 
\leq 2^A \tilde{\phi}_{\delta}(x) \int_{\tilde{\phi}_{\delta}(y) \leq 2^{A} \tilde{\phi}_{\delta}(x)} \tilde{\phi}(x-y) dy 
\lesssim \tilde{\phi}_{\delta}(x) 
\]
(here we used $A > k$ so that $\tilde{\phi} \in L^1(\R^k)$), and
\[
\int_{\tilde{\phi}_{\delta}(y) > 2^{A} \tilde{\phi}_{\delta}(x)} \tilde{\phi}(x-y) \tilde{\phi}_{\delta}(y) dy 
\leq 2^A \tilde{\phi}_{\delta}(x) \int_{\tilde{\phi}_{\delta}(y) > 2^{A} \tilde{\phi}_{\delta}(x)} \tilde{\phi}_{\delta}(y) dy 
\lesssim \tilde{\phi}_{\delta}(x),
\]
which holds because when $\tilde{\phi}_{\delta}(y) > 2^{A} \tilde{\phi}_{\delta}(x)$, one has $\tilde{\phi}_{\delta}(x) = \delta^{\frac{k(k+1)}{2}} (\delta^i |x_i|)^{-A}$ for some $i = 1, \dots, k$, with $(\delta^i |y_i|)^{-A} > 2^A (\delta^i |x_i|)^{-A}$, i.e. $|x_i| > 2 |y_i|$, so $|x_i-y_i| \geq \frac{|x_i|}{2}$, which implies $\tilde{\phi}(x-y) \leq |x_i-y_i|^{-A} \leq 2^A |x_i|^{-A} = 2^A \delta^{iA - \frac{k(k+1)}{2}} \tilde{\phi}_{\delta}(x) \leq 2^A \tilde{\phi}_{\delta}(x)$ (the last inequality used $A \geq \frac{k(k+1)}{2}$).
\end{proof}

\begin{lemma}
\label{lem:convolution-H}
Let $\Phi_{B}$ be as in Corollary~\ref{cor:torsion-local} with $A' \geq A+k-2$.
Then, for any $\nu \in (0,1]$, $I' \in \Part{\nu}$, $\xi' \in I'$, and $l \in \{1,\dots,k-1\}$, the following holds:
Let $H \subset \R^{k}$ be the $l$-dimensional subspace given by the orthogonal complement of $V^{k-l}(\xi')$ and $B_{H}(0,r)$ be the ball in $H$ centered at $0$ and of radius $r$.
Then 
\begin{equation}
\label{eq:10}
(\Phi_{B_H(0,\nu^{-(k-l+1)})} *_H \phi_{I'})(x) \lesssim \phi_{I'}(x)
\end{equation}
for every $x \in \R^k$ where $*_H$ denotes convolution on the subspace $H$.
\end{lemma}
\begin{proof}
By a change of variables, it suffices to show this in the case when $\xi' = 0$.
Since $\nu/2 \leq |I'| \leq \nu$, the desired inequality is equivalent to
\begin{equation}
\label{eq:9}
\int_{\R^l} \nu^{k'l} (1+\nu^{k'}|z|)^{-A'} (1+\sum_{i=1}^{k-l} \nu^i |x_i| + \sum_{j=k'}^k \nu^j |x_j-z_j|)^{-A}  dz
\lesssim (1+\sum_{i=1}^k \nu^i |x_i|)^{-A},
\end{equation}
where $k':=k-l+1$.
The left hand side is clearly bounded by $(1+\sum_{i=1}^{k-l} \nu^i |x_i|)^{-A}$, since we can drop the sum over $j$.
Also, for any $j_0 \in \{k', \ldots, k\}$, we have
\begin{equation}
\label{eq:12}
LHS\eqref{eq:9}
\lsm \int_{\R^l} \nu^{k'l} (1+\nu^{k'}|z|)^{-A'} (1+ \nu^{j_0} |x_{j_0}-z_{j_0}|)^{-A}  dz.
\end{equation}
Noting $|z| \sim |z_{j_0}| + |z'|$ where $z' \in \R^{l-1}$ is obtained from $z$ by dropping $z_{j_0}$, and integrating over $z'$ using the identity $\int_{\R^{l-1}} (M + |z'|)^{-A'} dz' = M^{-(A'-l+1)} \int_{\R^{l-1}} (1+|z'|)^{-A'} dz'$ for all $M > 0$, we see that
\begin{align}\label{eq:11}
\eqref{eq:12}
\lesssim \int_{\R} \nu^{k'} (1+\nu^{k'}|z_{j_0}|)^{-(A'-l+1)} (1+\nu^{j_0} |x_{j_0}-z_{j_0}|)^{-A}  dz_{j_0} \lesssim (1+  \nu^{j_0} |x_{j_0}|)^{-A}.
\end{align}
In the last inequality we used $A'-l+1 \geq A'-k+2 \geq A > 1$ and $\nu^{j_0} \leq \nu^{k'}$, which holds since $l \leq k-1$ and $j_0 \geq k'$ respectively,
and appealed to the dimension $k=1$ case of the inequality $\tilde{\phi}_{\delta_1} * \tilde{\phi}_{\delta_2} \lesssim \tilde{\phi}_{\delta_2}$ for all 
$0 < \delta_2 \leq \delta_1$. The latter is equivalent, via scaling, to the inequality \eqref{eq:phi_tilde_bdd} we proved earlier.
Since \eqref{eq:11} holds for any $j_0 \in \{k',\dots,k\}$, we have the desired estimate \eqref{eq:9}.
\end{proof}

\printbibliography

\end{document}